\documentclass[12pt,leqno,oneside]{amsart}
\usepackage{mathrsfs}
\usepackage{amsmath, amstext, amsthm, amssymb}
\usepackage{charter}
\usepackage{typearea}
\usepackage[hidelinks]{hyperref}
\usepackage{color}

\newtheorem{theorem}{Theorem}
\newtheorem{corollary}{Corollary}
\newtheorem{lemma}{Lemma}
\newtheorem{proposition}{Proposition}
\newtheorem*{theoremFS}{Theorem FS}
\newtheorem*{theoremR}{Theorem R}
\theoremstyle{definition}
\newtheorem{definition}{Definition}
\newtheorem{remark}{Remark}
\newtheorem{example}{Example}

\newcommand{\C}{\mathbb{C}}
\newcommand{\R}{\mathbb{R}}
\newcommand{\Z}{\mathbb{Z}}

\newcommand{\D}{\mathbb{D}}
\newcommand{\Q}{\mathbb{Q}}
\newcommand{\Qbar}{\overline{\mathbb{Q}}}
\newcommand{\Chat}{\widehat{\mathbb{C}}}
\newcommand{\Pc}{\mathrm{Pc}}
\newcommand{\Gal}{\operatorname{Gal}}
\newcommand{\dist}{\operatorname{dist}}
\newcommand{\capc}{\operatorname{cap}}
\newcommand{\supp}{\operatorname{supp}}

\begin{document}

\title{A Dynamical Fekete--Szeg\H{o} Theorem}
\author{Turgay Bayraktar}
\author{Mel\.{I}ke Efe}
\date{\today}
\address{Faculty of Engineering and Natural Sciences, Sabanc{\i} University, \.{I}stanbul, Turkey}
\email{tbayraktar@sabanciuniv.edu}
\email{melikeefe@sabanciuniv.edu}
\thanks{T. Bayraktar is partially supported by T\"{U}B\.{I}TAK grant ARDEB-1001/124F370}


\begin{abstract}
Let $E\subset\C$ be a compact set, symmetric with respect to the real axis and
regular for the Dirichlet problem. A classical theorem of Fekete and Szeg\H{o}
asserts that if $\capc(E)\ge 1$ then every neighbourhood of $E$ contains infinitely
many complete sets of conjugate algebraic integers. We prove a dynamical
strengthening: $\capc(E)\ge1$ if and only if every neighbourhood of the polynomially
convex hull $\Pc(E)$ contains filled Julia sets of monic integral polynomials of
arbitrarily large degree. When $\capc(E)=1$, we show that the minimal polynomials
$P_n$ of distinct algebraic integers whose entire Galois orbits approach $E$ have filled
Julia sets $K_{P_n}$ converging to $\Pc(E)$ in the Klimek metric, and Brolin
measures converging to the equilibrium measure $\mu_E$. The key step is that
algebraic numbers of small Rumely height $h_E$ have minimal polynomials that are
asymptotically minimal on $E$. As an arithmetic application, we show that the Call--Silverman
canonical heights $\hat h_{P_n}$ converge to $h_E$ uniformly on $\Qbar$.
\end{abstract}

\maketitle

\section{Introduction}

The interaction between logarithmic potential theory, arithmetic geometry and
polynomial dynamics has produced striking analogies between extremal problems in
approximation theory and equidistribution phenomena in arithmetic dynamics. The
purpose of this note is to develop a dynamical counterpart of a classical theorem
of Fekete and Szeg\H{o} \cite{FeketeSzego}: symmetric compact sets of capacity at
least one may be approximated not only by complete sets of conjugate algebraic
integers, but also by \emph{filled Julia sets} of polynomials with integer
coefficients.

\medskip

Let $E\subset\C$ be a compact set symmetric with respect to the real axis. By a
theorem of Fekete \cite{Fekete}, if $\capc(E)<1$ then some neighbourhood of $E$
contains only finitely many complete sets of conjugate algebraic integers;
conversely, Fekete and Szeg\H{o} \cite{FeketeSzego} proved that if
$\capc(E)\ge1$ then every neighbourhood of $E$ contains infinitely many such sets.
Such arithmetic approximation phenomena are closely related to logarithmic
potential theory. Bilu \cite{Bilu} showed that distinct nonzero algebraic numbers of small Weil
height have Galois orbits equidistributed with respect to the Haar measure on the
unit circle, and Rumely \cite{RumelyBilu} extended this to general compact sets:
for a symmetric compact set $E$ of capacity one, the equilibrium measure $\mu_E$ is
the weak limit of the Galois orbits of algebraic numbers whose degrees tend to infinity and whose
height $h_E$ tends to zero \cite[Theorem~1]{RumelyBilu}. Thus, from the arithmetic point of view, $\mu_E$
is the canonical limit distribution of small points relative to $E$.

\medskip

In our earlier work \cite{BE} we studied the dynamics of \emph{asymptotically
minimal polynomials} (see \S\ref{amp}). One of the main conclusions of \cite{BE} is
that, under suitable hypotheses on the zeros, asymptotic extremality forces
convergence of dynamical invariants: Brolin
measures converge to the equilibrium measure, and filled Julia sets converge in the
Klimek topology. The convergence mechanism itself is therefore already contained in
\cite{BE}. The new ingredients of the present note are the arithmetic input,
namely the passage from small relative height to asymptotic minimality
(Lemma~\ref{lem:height-to-supnorm}), and the realisation of the Rumely height by
integral polynomial dynamics (Corollary~\ref{cor:height-approx}).

\begin{theorem}
\label{Dynamical_FS}
Let $E\subset\C$ be a compact set symmetric with respect to the real axis
and regular for the Dirichlet problem. The following are equivalent:
\begin{enumerate}
\item[(i)] $\capc(E)\ge1$;
\item[(ii)] for every open neighbourhood $U$ of the polynomially convex hull
$\Pc(E)$ there exists a sequence of monic polynomials $P_n\in\Z[z]$ with
$\deg(P_n)\to\infty$ such that
\[
K_{P_n}\subset U \quad\text{for all sufficiently large } n .
\]
\end{enumerate}
\end{theorem}

Under the normalisation $\capc(E)=1$ we prove that the minimal polynomials of
algebraic integers all of whose conjugates approach $E$ generate filled Julia sets
converging to $\Pc(E)$. For a finite set $A\subset\C$ we write
\begin{equation}\label{eq:rho}
\rho(A,E):=\max_{\zeta\in A}\dist(\zeta,E),
\end{equation}
the \emph{directed} distance from $A$ to $E$. (The infimum distance
$\min_{\zeta\in A}\dist(\zeta,E)$ would not suffice; see Remark~\ref{rem:counterex}.)

\begin{theorem}
\label{Arithmetic_BE}
Let $E\subset\C$ be a compact set symmetric with respect to the real axis,
regular for the Dirichlet problem, with $\capc(E)=1$.
Let $\{\alpha_n\}$ be a sequence of pairwise distinct algebraic integers and let
$P_n\in\Z[z]$ be the minimal polynomial of $\alpha_n$. Assume that
\[
\rho_n:=\rho\bigl(\Gal(\alpha_n),E\bigr)\longrightarrow 0\quad (n\to\infty).
\]
Then $\deg P_n\to\infty$ and, after discarding finitely many terms of degree
less than two,
\begin{equation}\label{klimekconv}
\Gamma\bigl(K_{P_n},\Pc(E)\bigr)=\|g_{P_n}-g_E\|_{L^\infty(\C)}\longrightarrow 0,
\end{equation}
where $g_{P_n}$ is the dynamical Green function of $P_n$ and $g_E$ is the Green
function of $\Chat\setminus\Pc(E)$ with pole at $\infty$. In particular,
$\omega_{P_n}\to\mu_E$ in the weak$^*$ topology, where $\omega_{P_n}$ is the Brolin
measure (the unique measure of maximal entropy) of $P_n$.
\end{theorem}

Theorem~\ref{Arithmetic_BE} says that the filled Julia sets $K_{P_n}$ converge to
$\Pc(E)$ in the Klimek topology (\S\ref{Klimektop}). When $E\subset\R$, the set $E$
is polynomially convex, and Theorem~\ref{Arithmetic_BE} gives an approximation of
$E$ itself by filled Julia sets of integral polynomials in the Klimek topology;
in this case we also obtain Hausdorff convergence (Proposition~\ref{prop:hausdorff}).

For a symmetric regular compact set $E\subset\C$ of capacity one, Rumely \cite{RumelyBilu}
defines a height $h_E$ by replacing, at the Archimedean places, the function
$\log^+|z|$ in the Weil height by the Green function $g_E$. Our approximation result
admits a natural interpretation in arithmetic dynamics: the Rumely height is a
\emph{uniform} limit of canonical heights in the sense of Call and Silverman
\cite{CS} (see \eqref{dheight}) attached to integral polynomial dynamics.

\begin{corollary}
\label{cor:height-approx}
Let $E$, $\alpha_n$ and $P_n$ be as in Theorem~\ref{Arithmetic_BE}, and discard the
finitely many $n$ with $\deg P_n\le1$. Then
\[
\sup_{\alpha\in\Qbar}\bigl|\hat h_{P_n}(\alpha)-h_E(\alpha)\bigr|
\le \Gamma\bigl(K_{P_n},\Pc(E)\bigr)\longrightarrow 0,
\]
where $\hat h_{P_n}$ denotes the canonical height of $P_n$.
\end{corollary}

Since $\hat h_{P_n}$ vanishes on the preperiodic points of $P_n$, the uniform
estimate gives $h_E(\beta_n)\le\Gamma(K_{P_n},\Pc(E))\to0$ for \emph{any} choice of
preperiodic points $\beta_n\in\Qbar$ of $P_n$. If the $\beta_n$ are chosen pairwise
distinct, then $h(\beta_n)\le h_E(\beta_n)+C_E$ is bounded (see \eqref{eq:CE}), so
$\deg\beta_n\to\infty$ by Northcott's theorem, and Rumely's theorem
(Theorem~R below) shows that their Galois orbits equidistribute to $\mu_E$. Thus
the dynamical approximation furnishes a natural mechanism for producing algebraic
points of small $h_E$-height.

\section{Preliminaries}
\label{prelim}

\subsection{Logarithmic potential theory}
We use \cite{Ransford} as a general reference. Let $E\subset\C$ be a non-polar
compact set, let $\mu_E$ be its equilibrium measure and $\capc(E)$ its logarithmic
capacity. For a finite Borel measure $\nu$ with compact support write
$p_\nu(z):=\int\log|z-\zeta|\,d\nu(\zeta)$ for its logarithmic potential, so that
$\log\capc(E)=\iint\log|z-w|\,d\mu_E(z)\,d\mu_E(w)$.

Let $\Omega_E$ be the unbounded component of $\Chat\setminus E$ and
$\Pc(E):=\C\setminus\Omega_E$ its polynomially convex hull. The Green function
with pole at infinity, with its subharmonic extension to $\C$, satisfies
\begin{equation}\label{eq:green-potential}
g_E(z)=p_{\mu_E}(z)-\log\capc(E).
\end{equation}
It is nonnegative, harmonic and strictly positive on
$\Omega_E\setminus\{\infty\}$, and vanishes quasi-everywhere on $\Pc(E)$.
Moreover,
\begin{equation}\label{eq:green-infty}
g_E(z)=\log|z|-\log\capc(E)+o(1)\qquad(z\to\infty).
\end{equation}
We have $g_E=g_{\Pc(E)}$ and $\capc(E)=\capc(\Pc(E))$.
If $E$ is invariant under complex conjugation, then $g_E(\bar z)=g_E(z)$.

The set $E$ is called \emph{regular} if every point of $\partial\Omega_E$
is regular for the Dirichlet problem. Then $g_E$ is continuous on $\C$
and $\{g_E=0\}=\Pc(E)$; in particular \eqref{eq:green-potential} holds
with a continuous potential everywhere. These facts follow from Frostman's
theorem and the theory of Green functions; see \cite[Chapters~3--4]{Ransford}.
We also use monotonicity of Green functions: if $A\subset B$ are regular
nonpolar compact sets, then $g_B\le g_A$.
In particular, a finite union of nondegenerate closed disks is regular,
by comparison at each point with a constituent disk.
Finally, capacity is continuous along decreasing sequences of compact sets
\cite[Thm.~5.1.3(a)]{Ransford}.

\subsection{Polynomial dynamics and Brolin measures}
We refer to \cite{CG, Ransford}. Let $P(z)=a_dz^d+\dots\in\C[z]$ with $d\ge2$.
Its filled Julia set is
\[
K_P:=\{z\in\C:\ \{P^{\circ k}(z)\}_{k\ge0}\ \text{is bounded}\},
\]
where $P^{\circ k}$ denotes the $k$-th iterate, and $J_P:=\partial K_P$ is its
Julia set. $K_P$ is a regular polynomially convex compact set. The dynamical
Green function
\[
g_P(z):=\lim_{k\to\infty}\frac1{d^k}\log^+\bigl|P^{\circ k}(z)\bigr|
\]
exists, is continuous and subharmonic on $\C$, satisfies $g_P\circ P=d\,g_P$, and
coincides with the Green function of $\C\setminus K_P$ with pole at infinity
\cite[\S6.5]{Ransford}. Moreover
\[
g_P(z)=\log|z|+\frac1{d-1}\log|a_d|+o(1)\qquad (z\to\infty),
\]
so that, by \eqref{eq:green-infty},
\begin{equation}\label{eq:cap-julia}
\capc(K_P)=|a_d|^{-1/(d-1)} .
\end{equation}
The Brolin measure $\omega_P:=\Delta g_P$ (normalised Laplacian) is the equilibrium
measure of $K_P$, satisfies $\supp\omega_P=J_P$, $P_*\omega_P=\omega_P$ and
$d^{-1}P^*\omega_P=\omega_P$, and is the unique measure of maximal entropy.

\subsection{Klimek metric}\label{Klimektop}
Let $\mathcal{R}$ be the family of regular polynomially convex compact subsets of
$\C$. Klimek \cite{Klimek} defined
$\Gamma(E,F):=\max\{\sup_F g_E,\ \sup_E g_F\}$ for $E,F\in\mathcal R$.
By the domination principle this equals
\[
\Gamma(E,F)=\|g_E-g_F\|_{L^\infty(\C)}.
\]
$(\mathcal R,\Gamma)$ is a complete metric space \cite{Klimek}; we refer to
convergence in this metric as convergence in the \emph{Klimek topology}. Since
$g_E-g_F\to\log\capc(F)-\log\capc(E)$ at infinity,
\begin{equation}\label{eq:cap-lip}
\bigl|\log\capc(E)-\log\capc(F)\bigr|\le\Gamma(E,F).
\end{equation}
Filled Julia sets and, for regular $E$, the hull $\Pc(E)$ belong to $\mathcal R$.
If $P$ has degree $d\ge2$, the pull-back $F\mapsto P^{-1}(F)$ is a contraction of
$(\mathcal R,\Gamma)$ with unique fixed point $K_P$ \cite{Klimek}, so that
$\Gamma\bigl((P^{\circ k})^{-1}(F),K_P\bigr)\to0$ for every $F\in\mathcal R$.

\subsection{Asymptotically minimal polynomials}\label{amp}
\begin{definition}[{\cite[Def.~1.1]{BE}, case $p=\infty$}]
Let $E\subset\C$ be compact with $\capc(E)>0$. A sequence of polynomials
$P_n(z)=a_nz^{d_n}+\cdots$, $a_n\neq0$, with $d_n\to\infty$, is
\emph{asymptotically minimal on $E$} if
\[
\lim_{n\to\infty}\frac1{d_n}\log|a_n|=-\log\capc(E)
\qquad\text{and}\qquad
\lim_{n\to\infty}\frac1{d_n}\log\|P_n\|_{E}=0.
\]
\end{definition}

In \cite{BE} the index of the sequence is the degree. The proofs of
\cite[Thms.~1.2, 1.3]{BE} use the index only through the normalisation by the
degree and through passage to subsequences; they therefore apply verbatim to
sequences with degrees $d_n\to\infty$ (possibly repeated or with gaps), which is the
form used below.

We shall use the following two results:
\begin{enumerate}
\item If $\{P_n\}$ is asymptotically minimal on a nonpolar compact set $E$
  and its zeros are uniformly bounded, then $\omega_{P_n}\to\mu_E$ weakly
  \cite[Theorem~1.2]{BE}.
\item If, in addition, $E$ is regular and for every $\varepsilon>0$ all zeros
  eventually lie in the $\varepsilon$-neighbourhood of $\Pc(E)$, then
  $\Gamma(K_{P_n},\Pc(E))\to0$ \cite[Theorem~1.3]{BE}.
\end{enumerate}
Throughout, finitely many terms of degree less than two are discarded before
considering dynamical Green functions or Brolin measures.

\subsection{Arithmetic background}
For $\alpha\in\Qbar$ of degree $d$, its \emph{minimal polynomial}
$P_\alpha\in\Z[z]$ is the unique irreducible primitive polynomial with positive
leading coefficient $a_\alpha$ vanishing at $\alpha$; its roots
$\alpha_1,\dots,\alpha_d$ form the complete set $\Gal(\alpha)$ of conjugates of
$\alpha$. If $\alpha$ is an algebraic integer, $P_\alpha$ is monic.

\begin{theoremFS}[{Fekete \cite{Fekete}, Fekete--Szeg\H{o} \cite{FeketeSzego}}]
Let $E\subset\C$ be compact and symmetric with respect to the real axis. If
$\capc(E)\ge1$, every open neighbourhood of $E$ contains infinitely many complete
sets of conjugate algebraic integers. If $\capc(E)<1$, some neighbourhood of $E$
contains only finitely many.
\end{theoremFS}

\smallskip
\noindent\emph{Heights.} Let $K/\Q$ be a number field. For a non-Archimedean place
$v$ of $K$ above the prime $p$, let $|\cdot|_v$ be the absolute value extending
$|\cdot|_p$ and $N_v:=[K_v:\Q_p]$. For $\beta\in K$ the absolute logarithmic Weil
height is
\begin{equation}\label{eq:weil-height}
h(\beta)=\frac1{[K:\Q]}\Bigl(\sum_{\sigma:K\hookrightarrow\C}\log^+|\sigma(\beta)|
+\sum_{v\nmid\infty}N_v\log^+|\beta|_v\Bigr),
\end{equation}
which is independent of $K$ \cite[\S1.5]{BG}. Grouping the embeddings $\sigma$
into Archimedean places (one embedding for a real place, a conjugate pair for a
complex place, $N_v\in\{1,2\}$) gives the usual sum over all places. For
$\alpha\in\Qbar$ of degree $d$ one has \cite[Prop.~1.6.6]{BG}
\begin{equation}\label{eq:mahler}
h(\alpha)=\frac1d\Bigl(\log a_\alpha+\sum_{j=1}^d\log^+|\alpha_j|\Bigr).
\end{equation}

Let now $E\subset\C$ be compact, symmetric with respect to the real axis, with
$\capc(E)=1$. Following Rumely \cite{RumelyBilu}, for $\beta\in K$ set
\begin{equation}\label{Rdefn}
h_E(\beta)=\frac1{[K:\Q]}\Bigl(\sum_{\sigma:K\hookrightarrow\C}g_E(\sigma(\beta))
+\sum_{v\nmid\infty}N_v\log^+|\beta|_v\Bigr).
\end{equation}
Because $g_E(\bar z)=g_E(z)$, the two embeddings of a complex place contribute
equally; this is where the symmetry of $E$ is used, and it makes \eqref{Rdefn} a
sum over places. Since each conjugate $\alpha_j$ of $\alpha\in K$ equals
$\sigma(\alpha)$ for exactly $[K:\Q]/d$ embeddings $\sigma$, subtracting
\eqref{eq:weil-height} from \eqref{Rdefn} and using \eqref{eq:mahler} yields
\begin{equation}\label{Rheight}
h_E(\alpha)=\frac1d\Bigl(\log a_\alpha+\sum_{j=1}^dg_E(\alpha_j)\Bigr).
\end{equation}
In particular $h_E\ge0$, $h_E$ is independent of $K$, and $h_E=h$ when
$E=\overline{\D}$. If $E$ is regular, $g_E$ is continuous and
$g_E(z)-\log^+|z|\to0$ as $z\to\infty$ by \eqref{eq:green-infty}, so
\begin{equation}\label{eq:CE}
C_E:=\sup_{z\in\C}\bigl|g_E(z)-\log^+|z|\bigr|<\infty,
\qquad\text{hence}\qquad |h_E-h|\le C_E\ \text{on }\Qbar .
\end{equation}

\begin{remark}\label{rem:jensen}
Assume here that $E$ is regular. By \eqref{eq:green-potential} with $\capc(E)=1$, $g_E(\zeta)=\int\log|z-\zeta|\,d\mu_E(z)$
for every $\zeta\in\C$; summing over the roots of $P_\alpha$, \eqref{Rheight} becomes
the Jensen-type formula
\begin{equation}\label{eq:jensen}
h_E(\alpha)=\frac1d\int\log|P_\alpha(z)|\,d\mu_E(z),
\end{equation}
which for $E=\overline{\D}$ is the classical expression of $h(\alpha)$ through the
Mahler measure of $P_\alpha$. Formula \eqref{eq:jensen} is the source of the lower
bound in Lemma~\ref{lem:height-to-supnorm}.
\end{remark}

\begin{theoremR}[{Rumely \cite[Thm.~1]{RumelyBilu}}]
Let $E\subset\C$ be compact, symmetric with respect to the real axis, with
$\capc(E)=1$. If $\{\alpha_n\}\subset\Qbar$ satisfies $\deg\alpha_n\to\infty$ and
$h_E(\alpha_n)\to0$, then the measures
$\mu_n:=\frac1{\deg\alpha_n}\sum_{\zeta\in\Gal(\alpha_n)}\delta_\zeta$ converge
weakly to $\mu_E$.
\end{theoremR}

For $E=\overline{\D}$ this is Bilu's theorem \cite{Bilu}. Since $\mu_E$ is a
probability measure, weak convergence on $\C$ entails tightness of $\{\mu_n\}$;
hence $\mu_n\to\mu_E$ also weakly as measures on the compact space $\Chat$, with
$\mu_n(\{\infty\})=\mu_E(\{\infty\})=0$.

\section{Small height and asymptotic minimality}

\begin{lemma}
\label{lem:height-to-supnorm}
Let $E\subset\C$ be a compact set, symmetric with respect to the real axis, regular,
with $\capc(E)=1$. Let $\{\alpha_n\}\subset\Qbar$ be pairwise distinct with
$h_E(\alpha_n)\to0$, and let $P_n=a_nz^{d_n}+\cdots\in\Z[z]$ be the minimal
polynomial of $\alpha_n$. Then $d_n\to\infty$ and
\[
0\le h_E(\alpha_n)\le\frac1{d_n}\log\|P_n\|_E\longrightarrow0,
\qquad
0\le\frac1{d_n}\log a_n\le h_E(\alpha_n)\longrightarrow 0 .
\]
In particular $\{P_n\}$ is asymptotically minimal on $E$. No boundedness of the
conjugates is assumed.
\end{lemma}

\begin{proof}
Write $\alpha_{n,1},\dots,\alpha_{n,d_n}$ for the conjugates of $\alpha_n$,
$\mu_n:=d_n^{-1}\sum_j\delta_{\alpha_{n,j}}$ and
\[
u_n(z):=\int\log|z-\zeta|\,d\mu_n(\zeta)=\frac1{d_n}\log\frac{|P_n(z)|}{a_n},
\qquad b_n:=\int g_E\,d\mu_n .
\]

\smallskip\noindent\emph{Step 1: $d_n\to\infty$.}
By \eqref{eq:CE}, $h(\alpha_n)\le h_E(\alpha_n)+C_E$ is bounded. By Northcott's
theorem \cite[Thm.~1.6.8]{BG} there are only finitely many algebraic numbers of
bounded height and bounded degree; since the $\alpha_n$ are distinct, each degree
occurs only finitely often, i.e.\ $d_n\to\infty$.

\smallskip\noindent\emph{Step 2: leading coefficients and the mean of $g_E$.}
Since $a_n\ge1$ and $g_E\ge0$, formula \eqref{Rheight} gives
\begin{equation}\label{eq:two-terms}
0\le\frac1{d_n}\log a_n\le h_E(\alpha_n),\qquad 0\le b_n\le h_E(\alpha_n),
\end{equation}
so $d_n^{-1}\log a_n\to0$ and $b_n\to0$. By Step~1 and Theorem~R, $\mu_n\to\mu_E$
weakly on $\Chat$.

\smallskip\noindent\emph{Step 3: lower bound.}
Since $\mu_E$ is a probability measure supported on $E$, \eqref{eq:jensen} gives
\[
\log\|P_n\|_E\ \ge\ \int\log|P_n|\,d\mu_E\ =\ d_n\,h_E(\alpha_n)\ \ge\ 0 .
\]
All integrals are finite, $g_E$ being finite everywhere.

\smallskip\noindent\emph{Step 4: upper bound.}
Define on $E\times\Chat$ the renormalised kernel
\[
H(z,\zeta):=\log|z-\zeta|-g_E(\zeta)\quad(\zeta\in\C),\qquad H(z,\infty):=0 .
\]
On $E\times\C$, $H$ is upper semicontinuous, as the sum of an upper semicontinuous
function and a continuous one, with $H=-\infty$ exactly on the diagonal. As
$\zeta\to\infty$,
\[
H(z,\zeta)=\log\Bigl|1-\frac z\zeta\Bigr|-\bigl(g_E(\zeta)-\log|\zeta|\bigr)\longrightarrow0
\]
uniformly in $z\in E$, by boundedness of $E$ and \eqref{eq:green-infty} with
$\capc(E)=1$. Hence $H$ is upper semicontinuous on the compact metric space
$E\times\Chat$, and therefore bounded above. Moreover, for $z\in E$,
\[
\int H(z,\zeta)\,d\mu_n(\zeta)=u_n(z)-b_n,
\qquad
\int H(z,\zeta)\,d\mu_E(\zeta)=g_E(z)-\int g_E\,d\mu_E=0,
\]
where we used \eqref{eq:green-potential}, the symmetry $\log|z-\zeta|=\log|\zeta-z|$,
and $g_E\equiv0$ on $E\supset\supp\mu_E$.

The function $u_n$ is upper semicontinuous on $E$, so it attains its maximum at
some $z_n\in E$. This maximum is finite, because the non-polar set $E$ is infinite
and $P_n$ has finitely many zeros. Let $\{n_k\}$ be a subsequence with
$u_{n_k}(z_{n_k})\to\limsup_n\sup_Eu_n$, and pass to a further subsequence with
$z_{n_k}\to z\in E$. Then $\delta_{z_{n_k}}\otimes\mu_{n_k}\to\delta_z\otimes\mu_E$
weakly on $E\times\Chat$. For each $M>0$, the truncated kernel
$H_M:=\max\{H,-M\}$ is bounded and upper semicontinuous. Portmanteau gives
\[
\limsup_{k\to\infty}\bigl(u_{n_k}(z_{n_k})-b_{n_k}\bigr)
\le \int H_M(z,\zeta)\,d\mu_E(\zeta).
\]
Letting $M\to\infty$, monotone convergence applied after subtracting from
a common upper bound gives a right-hand side of
$\int H(z,\zeta)\,d\mu_E(\zeta)=0$.
Since $b_n\to0$, we obtain
$\limsup_n\sup_Eu_n\le0$. Consequently,
\[
\limsup_{n\to\infty}\frac1{d_n}\log\|P_n\|_E
=\limsup_{n\to\infty}\Bigl(\frac1{d_n}\log a_n+\sup_Eu_n\Bigr)\le0 .
\]

Steps~2--4 give the displayed estimates. Since $\capc(E)=1$, they are exactly the two
conditions of asymptotic minimality.
\end{proof}

\begin{remark}
The upper bound in Step~4 does not follow from weak convergence $\mu_n\to\mu_E$
alone. Logarithmic mass can escape to infinity: for normalised arclength $m$ on the
unit circle, $\nu_n=(1-\tfrac1n)m+\tfrac1n\delta_{e^n}\to m$ weakly, while
$p_{\nu_n}\to\log^+|z|+1\neq p_m$ locally uniformly on $\C$. The renormalisation
by $g_E$ in the kernel $H$ uses the information $b_n\to0$, which controls exactly
this logarithmic tail. Example~\ref{ex:unbounded} shows that the conjugates may
indeed be unbounded.
\end{remark}

\begin{remark}
In the setting of Theorem~\ref{Arithmetic_BE}, the equidistribution of
conjugates can also be obtained directly from the discriminant. Bounded
conjugates give bounded Weil heights, so distinctness and Northcott imply
$d_n\to\infty$. Since $P_n$ is monic and irreducible,
$0\ne\operatorname{Disc}(P_n)\in\Z$ and
\[
\sum_{i\ne j}\log|\alpha_{n,i}-\alpha_{n,j}|\ge0.
\]
On a fixed compact disk containing the conjugates, put
$k_M(z,w):=\max\{\log|z-w|,-M\}$, including the value $-M$ on the diagonal.
This kernel is bounded and continuous, and
\[
\iint k_M\,d\mu_n\,d\mu_n\ge-\frac{M}{d_n}.
\]
For any weak limit $\nu$ of the $\mu_n$, first let $n\to\infty$ along the
convergent subsequence and then $M\to\infty$. It follows that
$\iint\log|z-w|\,d\nu(z)\,d\nu(w)\ge0$.
The all-conjugates hypothesis implies $\supp\nu\subset E$. Since
$\log\capc(E)=0$, uniqueness of the equilibrium measure gives $\nu=\mu_E$.
\end{remark}

\begin{theorem}
\label{main thm}
Let $E\subset\C$ be compact, symmetric with respect to the real axis, regular, with
$\capc(E)=1$. Let $\{\alpha_n\}\subset\Qbar$ be pairwise distinct with
$h_E(\alpha_n)\to0$, and let $P_n\in\Z[z]$ be the minimal polynomial of $\alpha_n$.
Assume that the roots of all $P_n$ lie in a fixed disc $\overline{D(0,R)}$. Then
$\omega_{P_n}\to\mu_E$ weak$^*$.
\end{theorem}

\begin{proof}
By Lemma~\ref{lem:height-to-supnorm}, $d_n\to\infty$ (so $d_n\ge2$ for large $n$)
and $\{P_n\}$ is asymptotically minimal on $E$. Since the zeros are uniformly
bounded, \cite[Theorem~1.2]{BE} applies.
\end{proof}

The hypothesis of bounded roots is not automatic.

\begin{example}\label{ex:unbounded}
Let $E=\overline{\D}$, so that $h_E=h$. For $d\ge2$ and an integer $N\ge3$ put
$f(x)=x^d-Nx^{d-1}+1$. On $|z|=1$ we have $|Nz^{d-1}|=N>2\ge|z^d+1|$, so by
Rouch\'e's theorem $f$ has $d-1$ zeros in $|z|<1$ and exactly one zero $r$ in
$|z|>1$. Since $f(N)=1>0$ and $f(N-1)=1-(N-1)^{d-1}<0$, this zero is real and
$r\in(N-1,N)$.

\emph{$f$ is irreducible over $\Q$.} By Gauss's lemma a factorisation over $\Q$
would give $f=gh$ with $g,h\in\Z[x]$ monic of positive degree. One factor, say
$g$, does not vanish at $r$, so all roots of $g$ lie in $|z|<1$. Hence
$0<|g(0)|<1$, where $g(0)\neq0$ because $f(0)=1$. This contradicts $g(0)\in\Z$.

Thus every root $\alpha$ of $f$ has $f$ as its minimal polynomial, and by
\eqref{eq:mahler}
\[
\frac{\log(N-1)}d<h(\alpha)=\frac{\log r}{d}<\frac{\log N}d .
\]
Taking $N=N_d=\lfloor e^{\sqrt d}\rfloor$ for $d$ large, we obtain pairwise distinct
algebraic integers $\alpha_d$ (their degrees differ) with $h(\alpha_d)\to0$,
whereas $\max_{\sigma}|\sigma(\alpha_d)|>N_d-1\to\infty$.
\end{example}

\begin{theorem}
\label{thm klimek}
Let $E\subset\C$ be compact, symmetric with respect to the real axis, regular, with
$\capc(E)=1$. Let $\{\alpha_n\}\subset\Qbar$ be pairwise distinct with
$h_E(\alpha_n)\to0$, and let $P_n\in\Z[z]$ be the minimal polynomial of $\alpha_n$.
Assume that for every $\varepsilon>0$ there exists $N$ such that for all $n\ge N$
all zeros of $P_n$ lie in the $\varepsilon$-neighbourhood of $\Pc(E)$. Then
\[
\Gamma\bigl(K_{P_n},\Pc(E)\bigr)\longrightarrow0 .
\]
\end{theorem}

\begin{proof}
By Lemma~\ref{lem:height-to-supnorm}, $\{P_n\}$ is asymptotically minimal on $E$ and
$d_n\to\infty$. The hypothesis on the zeros is that of \cite[Theorem~1.3]{BE}, which
yields the claim.
\end{proof}

\section{Proofs of the main results}

\begin{proof}[Proof of Theorem~\ref{Arithmetic_BE}]
Since $\alpha_n$ is an algebraic integer, $P_n$ is monic. Let
\[
\omega_E(\delta):=\sup\{g_E(z):\ \dist(z,E)\le\delta\}.
\]
As $g_E$ is uniformly
continuous on the compact set $\{\dist(z,E)\le1\}$ and vanishes on $E$, we have
$\omega_E(\delta)\to0$ as $\delta\to0$. By \eqref{Rheight} with $a_{\alpha_n}=1$,
\[
0\le h_E(\alpha_n)=\frac1{d_n}\sum_{j=1}^{d_n}g_E(\alpha_{n,j})\le\omega_E(\rho_n)\longrightarrow0 .
\]
Moreover, all zeros of $P_n$ lie in $\{\dist(z,E)\le\rho_n\}\subset
\{\dist(z,\Pc(E))\le\rho_n\}$ with $\rho_n\to0$. Theorem~\ref{thm klimek} gives
\eqref{klimekconv}, and Lemma~\ref{lem:height-to-supnorm} gives $d_n\to\infty$.

For the measures: set $\Gamma_n:=\Gamma(K_{P_n},\Pc(E))$. On $K_{P_n}$ we have
$g_E\le g_{P_n}+\Gamma_n=\Gamma_n$, so $K_{P_n}\subset\{g_E\le1\}$ for large $n$,
a fixed compact set. Uniform convergence $g_{P_n}\to g_E$ implies
$\omega_{P_n}=\Delta g_{P_n}\to\Delta g_E=\mu_E$ in the sense of distributions. For
probability measures supported in a fixed compact set this is equivalent to
weak$^*$ convergence.
\end{proof}

\begin{remark}\label{rem:counterex}
The directed distance in \eqref{eq:rho} cannot be replaced by the infimum
distance, and the conclusion of Theorem~\ref{main thm} does not imply Klimek
convergence. Let $E=\overline{\D}$ and $P_d(z)=z^d-3z^{d-1}+1$, $d\ge2$. As in
Example~\ref{ex:unbounded} (with $N=3$), $P_d$ is irreducible, with $d-1$ roots in
$\D$ and one root $r_d\in(2,3)$, so its roots $\alpha_d$ are distinct algebraic
integers with $h_E(\alpha_d)=d^{-1}\log r_d\to0$ and roots bounded by $3$.
Hence $\omega_{P_d}\to\mu_E$ by Theorem~\ref{main thm}, and
$\min_{\zeta\in\Gal(\alpha_d)}\dist(\zeta,E)=0$. However
$P_d(3)-3=-2<0<4^{d-1}-3=P_d(4)-4$, so $P_d$ has a fixed point $q_d\in(3,4)$. Then
$q_d\in K_{P_d}$, and
\[
\|g_{P_d}-g_E\|_{L^\infty(\C)}\ge g_E(q_d)=\log q_d>\log3 .
\]
(In fact $q_d-3=(q_d-1)q_d^{1-d}\to0$.)
\end{remark}

The following lemma supplies the capacity reduction in
Theorem~\ref{Dynamical_FS}.

\begin{lemma}[Capacity-one sets near $E$]
\label{lem:cap-one}
Let $E\subset\C$ be compact, symmetric with respect to the real axis, regular, with
$\capc(E)>1$, and let $U$ be an open neighbourhood of $\Pc(E)$. Then there is a
compact set $S$, symmetric with respect to the real axis and regular, with
$\capc(S)=1$ and $\Pc(S)\subset U$.
\end{lemma}

\begin{proof}
Choose $t>0$ such that $V:=\{g_E\le t\}\subset U$; these compact sublevel
sets decrease to $\Pc(E)$ as $t\downarrow0$. If $U=\C$, any $t>0$ will do.
The set $V$ is symmetric and polynomially convex. Indeed, a bounded component
of its complement would have boundary value $t$ for the harmonic function
$g_E$ but $g_E>t$ inside, contradicting the maximum principle.

Choose $R>0$ so that $\{z:\dist(z,E)\le R\}\subset\{g_E<t\}$, and a finite
conjugation-invariant set of centers $\{c_1,\ldots,c_m\}\subset E$ whose open
disks of radius $R$ cover $E$. For $0\le r\le R$, set
\[
A_r:=\bigcup_{j=1}^m\overline{D(c_j,r)}.
\]
Then $A_R\subset V$ and $\capc(A_R)\ge\capc(E)>1$; for $r>0$, $A_r$ is
regular and symmetric. We verify that its capacity varies continuously.
If $0<r<s\le R$, comparison with a constituent disk gives
\[
0\le g_{A_r}(z)\le\log(s/r)\qquad(z\in A_s).
\]
Indeed, choose $j$ with $|z-c_j|\le s$ and use
$g_{A_r}(z)\le\log^+(|z-c_j|/r)$.
By \eqref{eq:green-potential}, symmetry of the logarithmic kernel and
$g_{A_s}=0$ on $A_r$, we have
\[
\begin{split}
\int g_{A_r}\,d\mu_{A_s}
&=\iint\log|z-w|\,d\mu_{A_r}(w)\,d\mu_{A_s}(z)-\log\capc(A_r)\\
&=\log\capc(A_s)-\log\capc(A_r).
\end{split}
\]
These integrals are finite by continuity of the equilibrium potentials.
Consequently
\[
0\le\log\frac{\capc(A_s)}{\capc(A_r)}\le\log\frac{s}{r},
\]
which proves continuity at positive radii. As $r\downarrow0$, $A_r$ decreases
to the finite set of centers, so $\capc(A_r)\to0$ by continuity of capacity
from above. The intermediate value theorem yields $r_0\in(0,R)$ with
$\capc(A_{r_0})=1$. Taking $S=A_{r_0}$, we obtain
$\Pc(S)\subset\Pc(V)=V\subset U$, as required.
\end{proof}

\begin{proof}[Proof of Theorem~\ref{Dynamical_FS}]
(i)$\Rightarrow$(ii). Let $U\supset\Pc(E)$ be open.

\emph{Case 1: $\capc(E)=1$.} Let $U_n:=\{z:\dist(z,E)<1/n\}$, an open neighbourhood
of $E$. By Theorem~FS, each $U_n$ contains infinitely many complete sets of
conjugate algebraic integers, so we may choose inductively an algebraic integer
$\alpha_n\notin\{\alpha_1,\dots,\alpha_{n-1}\}$ with $\Gal(\alpha_n)\subset U_n$.
Then the $\alpha_n$ are pairwise distinct and $\rho(\Gal(\alpha_n),E)\le1/n\to0$.
Let $P_n$ be the (monic) minimal polynomial of $\alpha_n$. By
Theorem~\ref{Arithmetic_BE}, $\deg P_n\to\infty$ and
$\|g_{P_n}-g_E\|_{L^\infty(\C)}\to0$.

If $U=\C$ we are done. Otherwise $F:=\C\setminus U$ is closed, nonempty and disjoint
from $\Pc(E)$. Choose $R$ with $g_E>1$ on $\{|z|>R\}$, possible by
\eqref{eq:green-infty}. The continuous function $g_E$ is strictly positive on the
compact set $F\cap\overline{D(0,R)}$. Hence
\[
\delta:=\inf_Fg_E\ \ge\ \min\Bigl\{1,\ \min_{F\cap\overline{D(0,R)}}g_E\Bigr\}>0,
\]
with the convention $\min_\varnothing=+\infty$. For $n$ so large that
$\|g_{P_n}-g_E\|_\infty<\delta/2$, every $z\in F$ satisfies
$g_{P_n}(z)\ge g_E(z)-\delta/2\ge\delta/2>0$, so $z\notin K_{P_n}=\{g_{P_n}=0\}$.
Thus $K_{P_n}\subset U$.

\emph{Case 2: $\capc(E)>1$.} By Lemma~\ref{lem:cap-one} there is a regular compact
set $S$, symmetric with respect to the real axis, with $\capc(S)=1$ and
$\Pc(S)\subset U$. Applying Case~1 to $S$ and its neighbourhood $U$ gives monic
$P_n\in\Z[z]$ with $\deg P_n\to\infty$ and $K_{P_n}\subset U$ for all large $n$.

(ii)$\Rightarrow$(i). For $m\ge1$ let $L_m:=\{z:\dist(z,\Pc(E))\le1/m\}$, a compact
set, and $U_m$ its interior, which contains $\Pc(E)$. By (ii) there is a monic
$P\in\Z[z]$ of degree $\ge2$ with $K_P\subset U_m\subset L_m$. By
\eqref{eq:cap-julia}, $\capc(K_P)=1$, hence $\capc(L_m)\ge1$. Since
$L_m\downarrow\Pc(E)$, continuity of capacity along decreasing compact sets gives
$\capc(E)=\capc(\Pc(E))=\lim_m\capc(L_m)\ge1$.
\end{proof}

\begin{remark}\label{rem:cap>1-obstruction}
By \eqref{eq:cap-julia}, every $P\in\Z[z]$ of degree $d\ge2$ satisfies
$\capc(K_P)=|a_d|^{-1/(d-1)}\le1$, with equality if and only if $P$ is monic up to
sign. By \eqref{eq:cap-lip}, a sequence $K_{P_n}$ with $P_n\in\Z[z]$ can therefore
converge in the Klimek topology only to a set of capacity at most one. Capacity
below one is not excluded for integral polynomials: for $0<r<1$, $d_n\to\infty$ and
$a_n:=\lceil r^{-(d_n-1)}\rceil$, the polynomials $P_n(z)=a_nz^{d_n}$ have
$K_{P_n}=\overline{D(0,\rho_n)}$ with
$\rho_n=a_n^{-1/(d_n-1)}\to r$, so $\Gamma(K_{P_n},\overline{D(0,r)})=|\log(\rho_n/r)|\to0$.
In the \emph{monic} setting of Theorem~\ref{Arithmetic_BE}, however, the normalisation
$\capc(E)=1$ is forced. If $\capc(E)>1$, Klimek convergence is impossible by
\eqref{eq:cap-lip}, because $\capc(K_{P_n})=1$. If $\capc(E)<1$, Fekete's theorem
(Theorem~FS) shows that only finitely many complete conjugate sets of algebraic
integers lie in a suitable neighbourhood of $E$, so no sequence as in
Theorem~\ref{Arithmetic_BE} exists. 
\end{remark}

\begin{proposition}\label{prop:hausdorff}
In the setting of Theorem~\ref{Arithmetic_BE}, assume in addition $E\subset\R$.
Then $K_{P_n}\to E$ and $J_{P_n}\to E$ in the Hausdorff metric.
\end{proposition}

\begin{proof}
Since $\C\setminus E$ is connected, $\Pc(E)=E$. Write
$\Gamma_n:=\Gamma(K_{P_n},E)\to0$.

\emph{Outer inclusion.} On $K_{P_n}$ we have $g_E\le\Gamma_n$, i.e.\
$J_{P_n}\subset K_{P_n}\subset E_{\Gamma_n}$, where
$E_\eta:=\{z:g_E(z)\le\eta\}$ for $\eta\ge0$. These compact sets decrease
to $E$ as $\eta\downarrow0$, so
$\sup_{z\in E_\eta}\dist(z,E)\to0$ as $\eta\to0$.

\emph{Inner inclusion.} First, $\supp\mu_E=E$. Indeed, if $\mu_E(D(x,r))=0$ for
some $x\in E$ and $r>0$, then by \eqref{eq:green-potential} $g_E$ is harmonic and
$\ge0$ on $D(x,r)$ and vanishes at $x$, hence $g_E\equiv0$ there by the minimum
principle. This contradicts $g_E>0$ on $D(x,r)\setminus\R\subset\Omega_E$. Given
$\varepsilon>0$, cover $E$ by finitely many discs $D(x_i,\varepsilon)$ with
$x_i\in E$. Then $\mu_E(D(x_i,\varepsilon))>0$, and since
$\omega_{P_n}\to\mu_E$ weak$^*$ (Theorem~\ref{Arithmetic_BE}),
$\liminf_n\omega_{P_n}(D(x_i,\varepsilon))\ge\mu_E(D(x_i,\varepsilon))>0$. As
$\supp\omega_{P_n}=J_{P_n}$, for $n$ large every disc $D(x_i,\varepsilon)$ meets
$J_{P_n}\subset K_{P_n}$. Hence every point of $E$ lies within $2\varepsilon$ of
$J_{P_n}$.
\end{proof}

No general equivalence between Klimek and Hausdorff convergence is claimed; see
\cite[Example~4.3]{BE}.

\section{Canonical heights}

Let $K$ be a number field and $P\in K[z]$ of degree $d\ge2$. The canonical height
of Call and Silverman \cite{CS} (see also \cite[\S3.4]{Silverman2007}) is
\begin{equation}\label{dheight}
\hat h_P(\alpha):=\lim_{k\to\infty}\frac1{d^k}\,h\bigl(P^{\circ k}(\alpha)\bigr),
\qquad \alpha\in\overline K .
\end{equation}
It satisfies $\hat h_P(P(\alpha))=d\,\hat h_P(\alpha)$, $|\hat h_P-h|$ is bounded on
$\overline K$, and $\hat h_P$ vanishes exactly on the preperiodic points of $P$
\cite{CS}. For monic integral polynomials the local decomposition is elementary:

\begin{lemma}\label{lem:local}
Let $P\in\Z[z]$ be monic of degree $d\ge2$, $K$ a number field and $\alpha\in K$.
Then
\[
\hat h_P(\alpha)=\frac1{[K:\Q]}\Bigl(\sum_{\sigma:K\hookrightarrow\C}g_P(\sigma(\alpha))
+\sum_{v\nmid\infty}N_v\log^+|\alpha|_v\Bigr).
\]
\end{lemma}

\begin{proof}
Let $v\nmid\infty$. Since the coefficients of $P$ are $v$-adic integers and its
leading coefficient is $1$, the ultrametric inequality gives $|P(z)|_v=|z|_v^d$ if
$|z|_v>1$, and $|P(z)|_v\le1$ if $|z|_v\le1$. Hence $\log^+|P(z)|_v=d\log^+|z|_v$,
and by induction $d^{-k}\log^+|P^{\circ k}(\alpha)|_v=\log^+|\alpha|_v$ for every
$k$. (This is the good-reduction case; cf.\ \cite[\S5]{BH}.) Applying
\eqref{eq:weil-height} to $\beta=P^{\circ k}(\alpha)\in K$ and using
$\sigma(P^{\circ k}(\alpha))=P^{\circ k}(\sigma(\alpha))$, we get
\[
\frac{h(P^{\circ k}(\alpha))}{d^k}=\frac1{[K:\Q]}\Bigl(\sum_{\sigma}\frac1{d^k}\log^+\bigl|P^{\circ k}(\sigma(\alpha))\bigr|
+\sum_{v\nmid\infty}N_v\log^+|\alpha|_v\Bigr).
\]
The second sum is independent of $k$ and has finitely many nonzero terms. The first
is a finite sum, and each term converges to $g_P(\sigma(\alpha))$. Letting
$k\to\infty$ proves the lemma, without any interchange of limit and infinite sum.
\end{proof}

\begin{proof}[Proof of Corollary~\ref{cor:height-approx}]
Fix $\alpha\in\Qbar$ and a number field $K\ni\alpha$. By Lemma~\ref{lem:local} and
\eqref{Rdefn}, the non-Archimedean terms of $\hat h_{P_n}(\alpha)$ and $h_E(\alpha)$
coincide, so
\[
\hat h_{P_n}(\alpha)-h_E(\alpha)=\frac1{[K:\Q]}\sum_{\sigma:K\hookrightarrow\C}
\bigl(g_{P_n}(\sigma(\alpha))-g_E(\sigma(\alpha))\bigr).
\]
There are exactly $[K:\Q]$ embeddings, so
$|\hat h_{P_n}(\alpha)-h_E(\alpha)|\le\|g_{P_n}-g_E\|_{L^\infty(\C)}=\Gamma(K_{P_n},\Pc(E))$,
a bound independent of $\alpha$. It tends to $0$ by Theorem~\ref{Arithmetic_BE}.
\end{proof}

\begin{remark}
For the monic integral polynomials considered here, all finite-place local heights
agree, so uniform convergence of the Archimedean Green functions yields uniform
convergence of the global heights. Thus Corollary~\ref{cor:height-approx} shows that the Rumely
height of every symmetric, regular compact set of capacity one lies in the
closure, for the topology of uniform convergence on $\Qbar$, of the canonical
heights of monic integral polynomials. Monicity is what makes this proof immediate:
it makes the local heights at all finite places equal to $\log^+|\cdot|_v$. For the
non-monic polynomials of Theorem~\ref{thm klimek} the Archimedean Green functions
still converge, but the finite places require a separate treatment before a
statement about global heights can be made. The uniform Green
function convergence used here comes from Theorem~\ref{Arithmetic_BE} (via
Theorem~\ref{thm klimek}), not from Theorem~\ref{main thm}; by
Remark~\ref{rem:counterex}, convergence of Brolin measures alone does not suffice.
\end{remark}



\begin{thebibliography}{99}

\bibitem{BH}
M.~H.~Baker and L.-C.~Hsia,
\newblock Canonical heights, transfinite diameters, and polynomial dynamics,
\newblock {\em J. Reine Angew. Math.} \textbf{585} (2005), 61--92.

\bibitem{BE}
T.~Bayraktar and M.~Efe,
\newblock On dynamics of asymptotically minimal polynomials,
\newblock {\em J.\ Approx.\ Theory} \textbf{295} (2023), Paper No.~105956.

\bibitem{Bilu}
Y.~Bilu,
\newblock Limit distribution of small points on algebraic tori,
\newblock {\em Duke Math.\ J.} \textbf{89} (1997), no.~3, 465--476.

\bibitem{BG}
E.~Bombieri and W.~Gubler,
\newblock {\em Heights in Diophantine Geometry},
\newblock New Mathematical Monographs \textbf{4}, Cambridge Univ. Press, Cambridge, 2006.

\bibitem{CS}
G.~S.~Call and J.~H.~Silverman,
\newblock Canonical heights on varieties with morphisms,
\newblock {\em Compositio Math.} \textbf{89} (1993), no.~2, 163--205.

\bibitem{CG}
L.~Carleson and T.~W.~Gamelin,
\newblock {\em Complex Dynamics},
\newblock Universitext, Springer-Verlag, New York, 1993.

\bibitem{Fekete}
M.~Fekete,
\newblock \"Uber die Verteilung der Wurzeln bei gewissen algebraischen Gleichungen
mit ganzzahligen Koeffizienten,
\newblock {\em Math.\ Z.} \textbf{17} (1923), 228--249.

\bibitem{FeketeSzego}
M.~Fekete and G.~Szeg\H{o},
\newblock On algebraic equations with integral coefficients whose roots belong to a given point set,
\newblock {\em Math.\ Z.} \textbf{63} (1955), 158--172.

\bibitem{Klimek}
M.~Klimek,
\newblock Metrics associated with extremal plurisubharmonic functions,
\newblock {\em Proc.\ Amer.\ Math.\ Soc.} \textbf{123} (1995), no.~9, 2763--2770.

\bibitem{Ransford}
T.~Ransford,
\newblock {\em Potential Theory in the Complex Plane},
\newblock London Mathematical Society Student Texts \textbf{28}, Cambridge University Press, Cambridge, 1995.

\bibitem{RumelyBilu}
R.~Rumely,
\newblock On Bilu's equidistribution theorem,
\newblock in: {\em Spectral Problems in Geometry and Arithmetic} (Iowa City, IA, 1997),
Contemp.\ Math.\ \textbf{237}, Amer.\ Math.\ Soc., Providence, RI, 1999, 159--166.

\bibitem{Silverman2007}
J.~H.~Silverman,
\newblock {\em The Arithmetic of Dynamical Systems},
\newblock Graduate Texts in Mathematics \textbf{241}, Springer, New York, 2007.

\end{thebibliography}
\end{document}